\documentclass[a4paper,12pt]{article}
\usepackage{times, url}
\usepackage{amsmath,amssymb,amsthm,amsfonts}

\newtheorem{theorem}{Theorem}
\newtheorem{lemma}{Lemma}
\begin{document}
\setcounter{page}{1}

\begin{center}
{\LARGE \bf  On the distribution of consecutive square-free numbers of the form $\mathbf{[\alpha n], [\alpha n]+1}$}
\vspace{8mm}

{\large \bf S. I. Dimitrov}
\vspace{3mm}

Faculty of Applied Mathematics and Informatics, Technical University of Sofia \\
8, St.Kliment Ohridski Blvd. 1756 Sofia, BULGARIA \\
e-mail: \url{sdimitrov@tu-sofia.bg}
\vspace{2mm}
\end{center}
\vspace{10mm}

\noindent
{\bf Abstract:}
In the present paper we show that there exist infinitely many consecutive square-free numbers of the form
$[\alpha n]$, $[\alpha n]+1$, where $\alpha>1$ is irrational number with bounded partial quotient
or irrational algebraic number.
\vspace{10mm}

\section{Notations}
\indent

Let $N$ be a sufficiently large positive integer. By $\varepsilon$ we denote an arbitrary small positive number,
not necessarily the same in different occurrences. We denote by $\mu(n)$ the M\"{o}bius function and by $\tau(n)$
the number of positive divisors of $n$. As usual $[t]$ and $\{t\}$ denote the integer part, respectively, the
fractional part of $t$. Let $||t||$ be the distance from $t$ to the nearest integer.
Instead of $m\equiv n\,\pmod {k}$ we write for simplicity $m\equiv n\,(k)$.
Moreover $e(t)$=exp($2\pi it$) and $\psi(t)=\{t\}-1/2$.
Let $\alpha>1$ be irrational number with bounded partial quotient or irrational algebraic number.

Denote
\begin{equation}\label{sigma}
\sigma=\prod\limits_{p}\left(1-\frac{2}{p^2}\right)\,.
\end{equation}
We define the characteristic function $\omega_\alpha(x)$ in the interval $(0, 1]$ as follows
\begin{equation}\label{omegaalphax}
\omega_\alpha(x)=\begin{cases}1\,,\;\; \mbox{ if }\, 1-\frac{1}{\alpha}<x<1\,;\\
\frac{1}{2}\,,\;\, \mbox{ if } x=1-\frac{1}{\alpha} \, \mbox{ or } \, x=1\,;\\
0\,,\;\; \mbox{  otherwise}\,;
 \end{cases}
\end{equation}
and we extend it periodically to all real line.
\newpage

\section{Introduction and statement of the result}
\indent

The problem for the consecutive square-free numbers arises in 1932 when Carlitz \cite{Carlitz} proved that
\begin{equation}\label{Carlitz}
\sum\limits_{n\leq N}\mu^2(n)\mu^2(n+1)=\sigma N+\mathcal{O}\big(N^{2/3+\varepsilon}\big)\,,
\end{equation}
where $\sigma$ is denoted by \eqref{sigma}.

Subsequently in 1949 Mirsky \cite{Mirsky} improved the error term of \eqref{Carlitz} to
\begin{equation}\label{Mirsky}
\sum\limits_{n\leq N}\mu^2(n)\mu^2(n+1)=\sigma N+\mathcal{O}\left(N^{2/3}(\log N)^{4/3}\right)\,.
\end{equation}
Further in 1984 Heath-Brown \cite{Heath-Brown} improved the error term of \eqref{Mirsky} to
\begin{equation}\label{Heath-Brown}
\sum\limits_{n\leq N}\mu^2(n)\mu^2(n+1)=\sigma N+\mathcal{O}\big(N^{7/11}(\log N)^7\big)\,.
\end{equation}
Finally in 2014 Reuss \cite{Reuss} improved the error term of \eqref{Heath-Brown} to
\begin{equation}\label{Reuss}
\sum\limits_{n\leq N}\mu^2(n)\mu^2(n+1)=\sigma N+\mathcal{O}\big(N^{(26+\sqrt{433})/81+\varepsilon}\big)
\end{equation}
and this is the best result up to now.

In 2008  G\"{u}lo\u{g}lu and  Nevans \cite{Guloglu} showed that there exist infinitely many square-free
numbers of the form $[\alpha n]$, where $\alpha>1$ is irrational number of finite type.
More precisely they proved that the asymptotic formula
\begin{equation*}
\sum\limits_{n\leq N}\mu^2([\alpha n])
=\frac{6}{\pi^2}N+\mathcal{O}\left(\frac{N\log\log N}{\log N}\right)
\end{equation*}
holds.

On the other hand in 2009 Abercrombie and Banks\cite{Abercrombie} showed
that for almost all $\alpha>1$ the asymptotic formula
\begin{equation*}
\sum\limits_{n\leq N}\mu^2([\alpha n])
=\frac{6}{\pi^2}N+\mathcal{O}\left(N^{\frac{2}{3}+\varepsilon}\right)
\end{equation*}
holds, however this result provides no specific value of $\alpha$.

Subsequently in 2013 Victorovich \cite{Victorovich} proved that when
$\alpha>1$ is irrational number with bounded partial quotient or irrational algebraic number,
then the asymptotic formula
\begin{equation*}
\sum\limits_{n\leq N}\mu^2([\alpha n])
=\frac{6}{\pi^2}N+\mathcal{O}\left(AN^{\frac{5}{6}}\log^5N\right)
\end{equation*}
holds. Here $A = A(N) = \max\limits_{1\leq m\leq N^2}\tau(m)$.

In 2018 the author \cite{Dimitrov1} showed that for any fixed $1<c<22/13$ there exist infinitely many
consecutive square-free numbers of the form $[n^c], [n^c]+1$.

Recently the author \cite{Dimitrov2} proved that there exist infinitely many consecutive square-free numbers of the form
$x^2+y^2+1$, $x^2+y^2+2$.

Define
\begin{equation}\label{SNalpha}
S(N, \alpha)=\sum\limits_{n\leq N}\mu^2([\alpha n]) \mu^2([\alpha n]+1)\,.
\end{equation}
Motivated by these results and following the method of Victorovich \cite{Victorovich}
we shall prove the following theorem.
\begin{theorem}
Let $\alpha>1$ be irrational number with bounded partial quotient or irrational algebraic number.
Then for the sum $S(N, \alpha)$  defined by \eqref{SNalpha} the asymptotic formula
\begin{equation}\label{asymptoticformula}
S(N, \alpha)=\sigma N+\mathcal{O}\left(N^{\frac{5}{6}+\varepsilon}\right)
\end{equation}
holds. Here $\sigma$ is defined by \eqref{sigma}.
\end{theorem}

\section{Lemmas}
\indent

\begin{lemma}\label{omega}
For the function $\omega_\alpha(x)$ defined by \eqref{omegaalphax} the formula
\begin{equation*}
\omega_\alpha(x)=\frac{1}{\alpha}+\psi(x)-\psi\left(x+\frac{1}{\alpha}\right)
\end{equation*}
holds.
\end{lemma}
\begin{proof}
See (\cite{Arkhipov}, p. 480 ).
\end{proof}

\begin{lemma}\label{Vaaler}
For every $J\geq2$, we have
\begin{equation*}
\psi(t)=\sum\limits_{1\leq|k|\leq J}a(k)e(kt)
+\mathcal{O}\Bigg(\sum\limits_{|k|\leq J}b(k)e(kt)\Bigg)\,,\quad
a(k)\ll1/|k|\,,\;\;b(k)\ll1/J\,.
\end{equation*}
\end{lemma}
\begin{proof}
See \cite{Vaaler}.
\end{proof}

\begin{lemma}\label{Trsum1}
If $ X\ge 1 $, then
\begin{equation*}
\Big|\sum_{ n\le X}e(\alpha n)\Big|
\le \min \left(X, \frac{1} { 2||\alpha||} \right) \,.
\end{equation*}
\end{lemma}
\begin{proof}
See (\cite{Karat}, Ch. 6, \S 2).
\end{proof}

\begin{lemma}\label{Trsum2}
Suppose that $X, Y\geq1$,\, $\lambda=\frac{a}{q}+\frac{\theta}{q^2}$,\, $q\geq1$,\,
$(a, q)=1$,\, $|\theta|\leq1$.
Then
\begin{equation*}
\sum_{n\le X} \, \min \left( Y,\, \frac{1}{ ||\lambda n|| } \right)
\ll\frac{XY}{q}+(X+q)\log 2q\,.
\end{equation*}
\end{lemma}
\begin{proof}
See (\cite{Vaughan}, Lemma 1).
\end{proof}

\section{Proof of the Theorem}
\indent

The equality $m = [\alpha n]$ is tantamount to $\alpha n-1 < m < \alpha n$,\,
$\frac{m}{\alpha}< n <\frac{m}{\alpha}+\frac{1}{\alpha}$,
i.e. $\left\{\frac{m}{\alpha}\right\}>1-\frac{1}{\alpha}$.
Then from  \eqref{SNalpha} we get
\begin{align}\label{SNalphaest1}
S(N, \alpha)&=\sum\limits_{n\leq N}\mu^2([\alpha n])\mu^2([\alpha n]+1)
=\sum\limits_{m\leq\alpha N\atop{\left\{\frac{m}{\alpha}\right\}>1-\frac{1}{\alpha}}}\mu^2(m)\mu^2(m+1)\nonumber\\
&=\sum\limits_{m\leq\alpha N}\mu^2(m)\mu^2(m+1)\omega_\alpha\left(\frac{m}{\alpha}\right)\,.
\end{align}
Now  \eqref{SNalphaest1} and Lemma \ref{omega} give us
\begin{align}\label{SNalphaest2}
S(N, \alpha)&=\frac{1}{\alpha}\sum\limits_{m\leq\alpha N}\mu^2(m)\mu^2(m+1)
+\sum\limits_{m\leq\alpha N}\mu^2(m)\mu^2(m+1)\left[\psi\left(\frac{m}{\alpha}\right)-\psi\left(\frac{m+1}{\alpha}\right)\right]\nonumber\\
&=\frac{1}{\alpha}S_1(N, \alpha)+S_2(N, \alpha)\,,
\end{align}
where
\begin{align}
\label{S1}
S_1(N, \alpha)&=\sum\limits_{m\leq\alpha N}\mu^2(m)\mu^2(m+1)\,,\\
\label{S2}
S_2(N, \alpha)&=\sum\limits_{m\leq\alpha N}\mu^2(m)\mu^2(m+1)\left[\psi\left(\frac{m}{\alpha}\right)-\psi\left(\frac{m+1}{\alpha}\right)\right]\,.
\end{align}
\textbf{Estimation of} $\mathbf{S_1(N, \alpha)}$

Bearing in mind \eqref{Reuss} and \eqref{S1} we obtain
\begin{equation}\label{S1est}
S_1(N, \alpha)=\sigma\alpha N+\mathcal{O}\big(N^{(26+\sqrt{433})/81+\varepsilon}\big)\,,
\end{equation}
where $\sigma$ is denoted by \eqref{sigma}.\\
\textbf{Estimation of} $\mathbf{S_2(N, \alpha)}$

Let
\begin{equation}\label{J}
J=\sqrt{\alpha N}\,.
\end{equation}
From \eqref{S2} and Lemma \ref{Vaaler} it follows
\begin{align}\label{S2est1}
S_2(N, \alpha)&=\sum\limits_{m\leq\alpha N}\mu^2(m)\mu^2(m+1)
\sum\limits_{1\leq|k|\leq J}a(k)e\left(\frac{km}{\alpha}\right)
\left[1-e\left(\frac{k}{\alpha}\right)\right]\nonumber\\
&+\mathcal{O}\Bigg(\sum\limits_{m\leq\alpha N}\mu^2(m)\mu^2(m+1)
\sum\limits_{|k|\leq J}b(k)e\left(\frac{km}{\alpha}\right)
\left[1+e\left(\frac{k}{\alpha}\right)\right]\Bigg)\nonumber\\
&=\sum\limits_{1\leq|k|\leq J}a(k)\left[1-e\left(\frac{k}{\alpha}\right)\right]
\sum\limits_{m\leq\alpha N}\mu^2(m)\mu^2(m+1)e\left(\frac{km}{\alpha}\right)\nonumber\\
&+\mathcal{O}\Bigg(\sum\limits_{|k|\leq J}b(k)
\left[1+e\left(\frac{k}{\alpha}\right)\right]
\sum\limits_{m\leq\alpha N}\mu^2(m)\mu^2(m+1)
e\left(\frac{km}{\alpha}\right)\Bigg)\nonumber\\
&=S_3(N, \alpha)+\mathcal{O}\Big(S_4(N, \alpha)\Big)\,,
\end{align}
where
\begin{align}
\label{S3}
S_3(N, \alpha)&=\sum\limits_{1\leq|k|\leq J}a(k)\left[1-e\left(\frac{k}{\alpha}\right)\right]
\sum\limits_{m\leq\alpha N}\mu^2(m)\mu^2(m+1)e\left(\frac{km}{\alpha}\right)\,,\\
\label{S4}
S_4(N, \alpha)&=\sum\limits_{|k|\leq J}b(k)\left[1+e\left(\frac{k}{\alpha}\right)\right]
\sum\limits_{m\leq\alpha N}\mu^2(m)\mu^2(m+1)e\left(\frac{km}{\alpha}\right)\,.
\end{align}
Using \eqref{S3} and Lemma \ref{Vaaler} we find
\begin{equation}\label{S3est1}
S_3(N, \alpha)\ll\sum\limits_{1\leq|k|\leq J}\frac{1}{k}
\left|\sum\limits_{m\leq\alpha N}\mu^2(m)\mu^2(m+1)e\left(\frac{km}{\alpha}\right)\right|\,.
\end{equation}
From \eqref{J}, \eqref{S4} and Lemma \ref{Vaaler} we get
\begin{equation}\label{S4est1}
S_4(N, \alpha)\ll\frac{1}{J}\sum\limits_{1\leq|k|\leq J}
\left|\sum\limits_{m\leq\alpha N}\mu^2(m)\mu^2(m+1)e\left(\frac{km}{\alpha}\right)\right|+\sqrt{N}\,.
\end{equation}
In order to estimate the sums $S_3(N, \alpha)$ and $S_4(N, \alpha)$ we shall prove the following lemma.
\begin{lemma}\label{Mainlemma}
Let $\alpha>1$ be irrational number with bounded partial quotient or irrational algebraic number.
Then for the sum
\begin{equation}\label{Sigma}
\Sigma=\sum\limits_{1\leq k\leq J}\frac{1}{k}
\left|\sum\limits_{m\leq\alpha N}\mu^2(m)\mu^2(m+1)e(\lambda km)\right|\,.
\end{equation}
where $\lambda=\frac{1}{\alpha}$\,, the estimation
\begin{equation*}
\Sigma\ll N^{\frac{5}{6}+\varepsilon}\,,
\end{equation*}
holds.
\end{lemma}
\begin{proof}
Using \eqref{Sigma} and the well-known identity $\mu^2(m)=\sum_{d^2|m}\mu(d)$ we write
\begin{align*}
\Sigma&=\sum\limits_{1\leq k\leq J}\frac{1}{k}
\left|\sum\limits_{m\leq\alpha N}\left(\sum\limits_{d^2|m}\mu(d)\right)\left(\sum\limits_{t^2|m+1}\mu(t)\right)
e(\lambda km)\right|\\
&=\sum\limits_{1\leq k\leq J}\frac{1}{k}
\left|\sum\limits_{d\leq\sqrt{\alpha N}}\mu(d)
\sum\limits_{t\leq\sqrt{\alpha N+1}}\mu(t)
\sum\limits_{r\leq\frac{\alpha N}{d^2}\atop{rd^2+1\equiv0\,(t^2)}}e(\lambda krd^2)\right|\\
&\leq\sum\limits_{1\leq k\leq J}\frac{1}{k}
\sum\limits_{d\leq\sqrt{\alpha N}}\sum\limits_{t\leq\sqrt{\alpha N+1}}
\left|\sum\limits_{r\leq\frac{\alpha N}{d^2}\atop{rd^2+1\equiv0\,(t^2)}}e(\lambda krd^2)\right|\,.
\end{align*}
Splitting the range of $k$, $d$ and $t$ into dyadic subintervals we obtain
\begin{equation}\label{Sigmaest1}
\Sigma\ll(\log N)^3\max\limits_{1\leq K\leq J/2\atop{1\leq D\leq\sqrt{\alpha N}/2\atop{1\leq T\leq\sqrt{\alpha N+1}/2}}}
\Sigma_0(K, D, T)\,,
\end{equation}
where
\begin{equation}\label{Sigma0}
\Sigma_0(K, D, T)=\sum\limits_{K\leq k\leq 2K}\frac{1}{k}
\sum\limits_{D\leq d\leq2D}\sum\limits_{T\leq t\leq2T}
\left|\sum\limits_{r\leq\frac{\alpha N}{d^2}\atop{rd^2+1\equiv0\,(t^2)}}e(\lambda krd^2)\right|\,.
\end{equation}
If $(d,\,t)>1$ then the sum $\Sigma_0(K, D, T)$ is empty. Suppose now that $(d,\,t)=1$.
Then the congruence $rd^2+1\equiv0\,(t^2)$ is equivalent to $r\equiv r_0\,(t^2)$, where $r_0$ is some integer with
$1\leq r_0\leq t^2$. From the last consideration and \eqref{Sigma0} it follows
\begin{equation}\label{Sigma0est1}
\Sigma_0(K, D, T)\leq\sum\limits_{K\leq k\leq 2K}\frac{1}{k}
\sum\limits_{D\leq d\leq2D}\sum\limits_{T\leq t\leq2T}
\left|\sum\limits_{s\leq\frac{\alpha N-r_0 d^2}{d^2 t^2}}e(\lambda k s d^2 t^2)\right|\,.
\end{equation}
Consider two cases.

\textbf{Case 1.}\, $DT\leq{(\alpha N)}^{1/6}$.

The inequality \eqref{Sigma0est1} and Lemma \ref{Trsum1} give us
\begin{align}\label{Sigma0est2}
\Sigma_0(K, D, T)&\ll\sum\limits_{K\leq k\leq 2K}\frac{1}{k}\sum\limits_{D\leq d\leq2D}\sum\limits_{T\leq t\leq2T}
\min\left(\frac{\alpha N}{d^2 t^2}, \frac{1}{||\lambda k d^2 t^2||} \right) \nonumber\\
&\leq\sum\limits_{K\leq k\leq 2K}\sum\limits_{D\leq d\leq2D}\sum\limits_{T\leq t\leq2T}
\min\left(\frac{\alpha N}{k d^2 t^2}, \frac{1}{||\lambda k d^2 t^2||} \right) \nonumber\\
&\leq\sum\limits_{K\leq k\leq 2K}\sum\limits_{D\leq d\leq2D}\sum\limits_{T\leq t\leq2T}
\min\left(\frac{\alpha N}{K D^2 T^2}, \frac{1}{||\lambda k d^2 t^2||} \right)\,.
\end{align}
Replacing $m=k d^2 t^2$ from \eqref{Sigma0est2} we get
\begin{align}\label{Sigma0est3}
\Sigma_0(K, D, T)&\leq\sum\limits_{K D^2 T^2\leq m\leq 32 K D^2 T^2}\Bigg(\sum\limits_{d^2 | m\atop{ d\leq2D}}1\Bigg)
\Bigg(\sum\limits_{t^2 | m\atop{ t\leq2T}}1\Bigg)
\min\left(\frac{\alpha N}{K D^2 T^2}, \frac{1}{||\lambda m||}\right) \nonumber\\
&\leq\sum\limits_{K D^2 T^2\leq m\leq 32 K D^2 T^2}\tau^2(m)
\min\left(\frac{\alpha N}{K D^2 T^2}, \frac{1}{||\lambda m||}\right) \nonumber\\
&\ll N^\varepsilon\sum\limits_{K D^2 T^2\leq m\leq 32 K D^2 T^2}
\min\left(\frac{\alpha N}{K D^2 T^2}, \frac{1}{||\lambda m||}\right)\,.
\end{align}
Since $\alpha$ \big(therefore $\lambda=\frac{1}{\alpha}$\big)
is irrational number with bounded partial quotient or irrational algebraic number
then $\lambda$ can be represented in the form
$\lambda=\frac{a}{q}+\frac{\theta}{q^2}$, $(\alpha N)^{\frac{1}{2}-\varepsilon}\ll q\ll(\alpha N)^{\frac{1}{2}}$,
$(a, q)=1$, $|\theta|\leq1$.
This follows for example from (\cite{Victorovich}, Ch.2, Lemma 1.5, Lemma 1.6).
Bearing in mind these considerations, \eqref{J}, \eqref{Sigma0est3}, Lemma \ref{Trsum2},
the inequalities $DT\leq{(\alpha N)}^{1/6}$ and $K\leq J$ we find
\begin{equation}\label{Sigma0est4}
\Sigma_0(K, D, T)\ll N^\varepsilon\left(\frac{\alpha N}{q}+K D^2 T^2+q\right)\log N\ll N^{\frac{5}{6}+\varepsilon}\,.
\end{equation}
\newpage

\textbf{Case 2.}\, $DT>{(\alpha N)}^{1/6}$.

Using \eqref{J}, \eqref{Sigma0est1}, the trivial estimate,
the inequalities $DT>{(\alpha N)}^{1/6}$ and $K\leq J$ we obtain
\begin{equation}
\label{Sigma0est5}
\Sigma_0(K, D, T)\leq\sum\limits_{K\leq k\leq 2K}\frac{1}{k}
\sum\limits_{D\leq d\leq2D}\sum\limits_{T\leq t\leq2T}\frac{\alpha N}{d^2 t^2}
\ll\frac{\alpha N}{DT}\log K\ll N^{\frac{5}{6}+\varepsilon}\,.
\end{equation}
From \eqref{Sigmaest1}, \eqref{Sigma0est4} and \eqref{Sigma0est5} it follows
\begin{equation*}
\Sigma\ll N^{\frac{5}{6}+\varepsilon}\,.
\end{equation*}
The lemma is proved.
\end{proof}
On the one hand \eqref{S3est1} and Lemma \ref{Mainlemma} give us
\begin{equation}\label{S3est2}
S_3(N, \alpha)\ll N^{\frac{5}{6}+\varepsilon}\,.
\end{equation}
On the other hand \eqref{S4est1} and Lemma \ref{Mainlemma} imply
\begin{equation}\label{S4est2}
S_4(N, \alpha)\ll\sum\limits_{1\leq|k|\leq J}\frac{1}{k}
\left|\sum\limits_{m\leq\alpha N}\mu^2(m)\mu^2(m+1)e\left(\frac{km}{\alpha}\right)\right|+\sqrt{N}
\ll N^{\frac{5}{6}+\varepsilon}\,.
\end{equation}
By \eqref{S2est1}, \eqref{S3est2} and \eqref{S4est2} we find
\begin{equation}\label{S2est2}
S_2(N, \alpha)\ll N^{\frac{5}{6}+\varepsilon}\,.
\end{equation}

\textbf{The end of the proof}

Bearing in mind \eqref{SNalphaest2}, \eqref{S1est} and  \eqref{S2est2}
we obtain the asymptotic formula \eqref{asymptoticformula}.

The theorem is proved.
\vspace{5mm}

\textbf{Acknowledgments.}
The author thanks Professor Stephen Choi for his helpful comments and suggestions,
that led to improvement of the reminder term in the asymptotic formula \eqref{asymptoticformula}.

\end{document}